\documentclass[10.5pt]
{amsart}
\usepackage{amssymb,amsthm,amsmath,latexsym}

\usepackage{float}
\usepackage{url}
\usepackage{systeme}
\usepackage{wrapfig}
\usepackage[margin=0.7in]{geometry}
\usepackage{tikz}
\usepackage{pgfplots}
\usepackage{xcolor,colortbl}
\pgfplotsset{compat=1.18}

\usepackage{array}

\definecolor{gold}{RGB}{218,165,32}

\usepackage{anyfontsize}
\fontsize{10.5}{13}\selectfont

\usepackage[justification=centering]{caption}
\usetikzlibrary{shapes,backgrounds,arrows.meta}
\usepackage{tikz-3dplot}

\usepackage{quoting}
\newtheorem*{proposition}{Proposition}

\newtheorem*{conjecture}{Conjecture}

\newtheorem*{remark*}{Remark}
\newtheorem*{conjecture*}{Conjecture}

\makeatletter
\renewcommand*\env@matrix[1][*\c@MaxMatrixCols c]{% 
  \hskip -\arraycolsep
  \let\@ifnextchar\new@ifnextchar
  \array{#1}}
\makeatother

\makeatletter
\renewcommand{\section}{\@startsection{section}{1}{\z@}%
  {-2ex \@plus -0.5ex \@minus -.2ex}%  <-- Space before the title (Default was -3.5ex)
  {0.8ex \@plus .2ex}%                 <-- Space after the title (Default was 2.3ex)
  {\normalfont\Large\bfseries}}
\makeatother

\usepackage{hyperref}

\begin{document}

\title
{\small Discovering Heron's Triangle Area Formula in a Morning in the Rain
}
\markright{ \, Heron's Formula---In A Morning }

\author{Eric L.~Grinberg}
\date{} % This suppresses the date

%:
\newenvironment{acknowledgement}%       New acknowledgement environment
    {\large\bfseries Acknowledgement%
    \par\medskip\normalfont\normalsize}%
    {}%

\begin{abstract}
Motivations and derivations for Heron's triangle area formula abound. But how would one \emph{discover} the formula? We recount the challenge of designing a laboratory experiment to discover Heron's formula, and of doing so in the morning, in time for a noon class---a snapshot of collaboration with AI in fall 2025. We then reflect on extensions and generalizations.
\end{abstract}

\maketitle

%\author{ Eric Grinberg }
%\address{UMass Boston }
%\email{eric.grinberg@umb.edu }

\vspace{-20pt}
\section{Introduction}

Heron's formula for triangle area by side lengths occurs early in the curriculum: 
\[ \operatorname{area}(\triangle abc)= \sqrt{s(s-a)(s-b)(s-c)}\, ; \,\,\, s \equiv \frac{1}{2}(a+b+c)
\text{\qquad  ($s$ is for \emph{semi-perimeter})}
.
\] 
Sometimes it is motivated, sometimes derived, with very good resources available in the literature. But Heron's original proof was complicated. How did Heron discover the formula? Maybe the result was known to Archimedes, and was already in the ether during Heron's time. But how was the formula discovered in the first place? Could we design a laboratory experiment aimed at discovering it, just as in science classes we perform experiments to ``discover" laws of nature? And could we complete the design in real-time, before a noon class?

\begin{wraptable}[24]{r}
{0.40\textwidth}
\vspace{-1.5\baselineskip}
\centering
\scriptsize
\caption{Random triangles, gold-standard areas.}
\label{tab:triangles}
\setlength{\tabcolsep}{3pt}

% ------------------------------------------------------------
% Colors
% ------------------------------------------------------------
\definecolor{parchment}{RGB}{238,224,191}
\definecolor{parchmentwarm}{RGB}{246,230,190}
\definecolor{tablegray}{RGB}{132,118,94}
\definecolor{rulegray}{RGB}{158,144,116}
\definecolor{figbrown}{RGB}{72,42,19}
\definecolor{auxbrown}{RGB}{142,126,98}

\begin{tikzpicture}

% ------------------------------------------------------------
% The actual table: this determines the footprint.
% ------------------------------------------------------------
\node[anchor=north west, inner sep=0pt, outer sep=0pt] (tbl)
at (0,0)
{%
  \begin{minipage}{0.40\textwidth}
  \centering
  \arrayrulecolor{rulegray}
  \color{tablegray}
  \begin{tabular}{|c|c|c|c|c|}
  \hline
  $i$ & $a$ & $b$ & $c$ & GoldArea \\
  \hline
   1 & 2.341178 & 2.191749 & 1.058027 & 1.154186 \\
   2 & 1.973164 & 1.862354 & 1.051102 & 0.963894 \\
   3 & 2.523576 & 2.514548 & 0.516247 & 0.646708 \\
   4 & 2.618736 & 2.009315 & 0.741791 & 0.482996 \\
   5 & 2.517821 & 2.324329 & 1.840570 & 2.049419 \\
   6 & 2.932789 & 1.880102 & 1.446336 & 1.138238 \\
   7 & 2.654267 & 2.573512 & 2.046299 & 2.459089 \\
   8 & 2.261430 & 1.943380 & 0.614561 & 0.546850 \\
   9 & 2.447087 & 1.878682 & 0.956349 & 0.811159 \\
  10 & 2.503369 & 1.625011 & 1.057538 & 0.587580 \\
  11 & 2.652336 & 1.669632 & 1.192051 & 0.700776 \\
  12 & 2.045532 & 2.336471 & 1.579709 & 1.586595 \\
  13 & 2.726012 & 2.301773 & 1.169669 & 1.332507 \\
  14 & 2.536970 & 2.372056 & 0.898894 & 1.066113 \\
  15 & 2.688913 & 2.137186 & 1.071063 & 1.080001 \\
  16 & 2.899991 & 2.206734 & 1.166717 & 1.166371 \\
  17 & 2.629956 & 1.854568 & 1.203127 & 0.993562 \\
  18 & 2.629244 & 2.442761 & 1.393906 & 1.684140 \\
  19 & 2.758313 & 2.290698 & 1.043668 & 1.152307 \\
  20 & 2.888780 & 1.976768 & 1.006820 & 0.507586 \\
  \hline
  \end{tabular}
  \end{minipage}%
};

% ------------------------------------------------------------
% Parchment background tied exactly to the table rectangle.
% Pure TikZ approximation: warm wash, faint fibres, and stains.
% ------------------------------------------------------------
\begin{scope}[on background layer]

  % Base parchment
  \fill[parchment]
    ([xshift=-0.8mm,yshift=0.8mm]tbl.north west)
    rectangle
    ([xshift=0.8mm,yshift=-0.8mm]tbl.south east);

  % Soft warm wash
  \fill[parchmentwarm, opacity=0.18]
    ([xshift=-0.8mm,yshift=0.8mm]tbl.north west)
    rectangle
    ([xshift=0.8mm,yshift=-0.8mm]tbl.south east);

  % Subtle edge darkening
  \draw[figbrown, opacity=0.08, line width=1.1mm]
    ([xshift=-0.45mm,yshift=0.45mm]tbl.north west)
    rectangle
    ([xshift=0.45mm,yshift=-0.45mm]tbl.south east);

  % Faint horizontal fibre lines
  \foreach \yy in {0.06,0.12,...,0.96}{
    \draw[figbrown, opacity=0.030, line width=0.04mm]
      ($(tbl.south west)!\yy!(tbl.north west)$)
      --
      ($(tbl.south east)!\yy!(tbl.north east)$);
  }

  % Faint vertical fibre lines
  \foreach \xx in {0.07,0.14,...,0.98}{
    \draw[figbrown, opacity=0.022, line width=0.04mm]
      ($(tbl.south west)!\xx!(tbl.south east)$)
      --
      ($(tbl.north west)!\xx!(tbl.north east)$);
  }

  % A few fixed pale stains, deterministic rather than random
  \fill[figbrown, opacity=0.030]
    ($(tbl.south west)!0.20!(tbl.north east)$) circle [radius=1.4mm];

  \fill[figbrown, opacity=0.022]
    ($(tbl.south west)!0.46!(tbl.north east)$) circle [radius=2.1mm];

  \fill[figbrown, opacity=0.026]
    ($(tbl.south west)!0.72!(tbl.north east)$) circle [radius=1.7mm];

  \fill[white, opacity=0.08]
    ($(tbl.south west)!0.34!(tbl.north east)$) circle [radius=1.8mm];

\end{scope}

% ------------------------------------------------------------
% Geometric overlay.
%
% IMPORTANT:
% x and y use the same physical unit length. This preserves
% Euclidean geometry: circles remain circles, perpendicularity is
% preserved, and the incircle remains inscribed.
% ------------------------------------------------------------
\begin{scope}[scale=0.80, transform shape,
  shift={(tbl.south west)},
  x=0.40\textwidth,
  y=0.40\textwidth,
  shift={(0.17,0.025)}
]

  % Original coherent schematic triangle:
  % A=(1.55,3.60), B=(0,0), Gamma=(4.15,0).
  % Uniform scale k=0.13, translated so Gamma is near x=0.93.
  %
  % Do not shear this scope and do not independently move points.

  \coordinate (A) at (0.5920,0.8680);
  \coordinate (B) at (0.3905,0.4000);
  \coordinate (G) at (0.9300,0.4000); % Gamma

  % Incenter and scaled incircle radius.
  \coordinate (H) at (0.6264,0.5552);
  \def\IncircleRadius{0.1552}

  % Tangency points: perpendicular feet from H to the three sides.
  \coordinate (D) at (0.4838,0.6166); % Delta on AB
  \coordinate (E) at (0.6264,0.4000); % E on B Gamma
  \coordinate (Z) at (0.7522,0.6461); % Z on Gamma A

  % Auxiliary points.
  \coordinate (T) at (0.1168,0.4000); % Theta
  \coordinate (L) at (0.3905,0.0939); % Lambda

  % ----------------------------------------------------------
  % Main triangle and extension.
  % ----------------------------------------------------------
  \draw[
    line width=0.50mm,
    color=figbrown,
    line cap=round,
    line join=round
  ]
    (A)--(B)--(G)--cycle;

  \draw[
    line width=0.20mm,
    color=figbrown,
    line cap=round
  ]
    (T)--(G);

  % ----------------------------------------------------------
  % Incircle: true circle, tangent to all three sides.
  % ----------------------------------------------------------
  \draw[
    line width=0.28mm,
    color=figbrown,
    line cap=round
  ]
    (H) circle [radius=\IncircleRadius];

  % ----------------------------------------------------------
  % Construction lines.
  % ----------------------------------------------------------
  \draw[line width=0.14mm, color=figbrown, line cap=round] (H)--(A);
  \draw[line width=0.14mm, color=figbrown, line cap=round] (H)--(B);
  \draw[line width=0.14mm, color=figbrown, line cap=round] (H)--(G);

  \draw[line width=0.14mm, color=figbrown, line cap=round] (H)--(D);
  \draw[line width=0.14mm, color=figbrown, line cap=round] (H)--(E);
  \draw[line width=0.14mm, color=figbrown, line cap=round] (H)--(Z);

  % ----------------------------------------------------------
  % Clearer auxiliary dashed construction.
  % Still secondary, but easier to see.
  % ----------------------------------------------------------
  \draw[
    line width=0.24mm,
    color=auxbrown!65!figbrown,
    dash pattern=on 2.2pt off 1.6pt,
    line cap=round
  ]
    (H)--(L);

  \draw[
    line width=0.24mm,
    color=auxbrown!65!figbrown,
    dash pattern=on 2.2pt off 1.5pt,
    line cap=round
  ]
    (B)--(L);

  \draw[
    line width=0.24mm,
    color=auxbrown!65!figbrown,
    dash pattern=on 2.2pt off 1.5pt,
    line cap=round
  ]
    (G)--(L);

  % ----------------------------------------------------------
  % Small right-angle markers at the tangency points.
  % Decorative/indicative only; geometry remains defined above.
  % ----------------------------------------------------------
  \draw[line width=0.08mm, color=figbrown]
    ($(D)!0.012!(A)$)--++(0.010,-0.006)--($(D)!0.012!(H)$);

  \draw[line width=0.08mm, color=figbrown]
    ($(E)!0.012!(B)$)--++(0,0.012)--($(E)!0.012!(H)$);

  \draw[line width=0.08mm, color=figbrown]
    ($(Z)!0.012!(G)$)--++(-0.009,-0.008)--($(Z)!0.012!(H)$);

  % ----------------------------------------------------------
  % Point markers.
  % ----------------------------------------------------------
  \foreach \P in {A,B,G,H,D,E,Z,T,L}
    \fill[figbrown] (\P) circle (0.45mm);

  % ----------------------------------------------------------
  % Labels with subtle parchment-colored halos.
  % ----------------------------------------------------------

  % A
  \node[text=parchment!85!white, font=\fontsize{14}{14}\selectfont\itshape]
    at ($(A)+(0.002,0.028)$) {$A$};
  \node[text=figbrown, font=\fontsize{14}{14}\selectfont\itshape]
    at ($(A)+(0,0.030)$) {$A$};

  % B
  \node[text=parchment!85!white, font=\fontsize{14}{14}\selectfont\itshape]
    at ($(B)+(-0.020,-0.012)$) {$B$};
  \node[text=figbrown, font=\fontsize{14}{14}\selectfont\itshape]
    at ($(B)+(-0.022,-0.010)$) {$B$};

  % Gamma
  \node[text=parchment!85!white, font=\fontsize{14}{14}\selectfont\itshape]
    at ($(G)+(0.020,-0.008)$) {$\Gamma$};
  \node[text=figbrown, font=\fontsize{14}{14}\selectfont\itshape]
    at ($(G)+(0.018,-0.006)$) {$\Gamma$};

  % H
  \node[text=parchment!85!white, font=\fontsize{14}{14}\selectfont\itshape]
    at ($(H)+(0.026,0.002)$) {$H$};
  \node[text=figbrown, font=\fontsize{14}{14}\selectfont\itshape]
    at ($(H)+(0.024,0.004)$) {$H$};

  % Delta
  \node[text=parchment!85!white, font=\fontsize{14}{14}\selectfont\itshape]
    at ($(D)+(-0.028,0.002)$) {$\Delta$};
  \node[text=figbrown, font=\fontsize{14}{14}\selectfont\itshape]
    at ($(D)+(-0.030,0.004)$) {$\Delta$};

  % E
  \node[text=parchment!85!white, font=\fontsize{14}{14}\selectfont\itshape]
    at ($(E)+(0.020,-0.014)$) {$E$};
  \node[text=figbrown, font=\fontsize{14}{14}\selectfont\itshape]
    at ($(E)+(0.018,-0.012)$) {$E$};

  % Z
  \node[text=parchment!85!white, font=\fontsize{14}{14}\selectfont\itshape]
    at ($(Z)+(0.030,0.002)$) {$Z$};
  \node[text=figbrown, font=\fontsize{14}{14}\selectfont\itshape]
    at ($(Z)+(0.028,0.004)$) {$Z$};

  % Theta
  \node[text=parchment!85!white, font=\fontsize{14}{14}\selectfont\itshape]
    at ($(T)+(-0.038,0.002)$) {$\Theta$};
  \node[text=figbrown, font=\fontsize{14}{14}\selectfont\itshape]
    at ($(T)+(-0.040,0.004)$) {$\Theta$};

  % Lambda
  \node[text=parchment!85!white, font=\fontsize{14}{14}\selectfont\itshape]
    at ($(L)+(-0.020,-0.016)$) {$\Lambda$};
  \node[text=figbrown, font=\fontsize{14}{14}\selectfont\itshape]
    at ($(L)+(-0.022,-0.014)$) {$\Lambda$};

\end{scope}

% ------------------------------------------------------------
% Small proof-without-words triangle, added near the upper left.
% This block is independent and does not alter the main Heron
% diagram or the background table.
%
% Placement knobs:
%   xshift=0.018\textwidth  moves it right/left.
%   yshift=-0.105\textwidth moves it down/up from the table's top.
%
% To move UP, make yshift less negative, e.g. -0.095\textwidth.
% To move DOWN, make yshift more negative, e.g. -0.115\textwidth.
% ------------------------------------------------------------
\begin{scope}
  % Clip to the table rectangle so the small triangle never protrudes.
  \clip
    ([xshift=-0.8mm,yshift=0.8mm]tbl.north west)
    rectangle
    ([xshift=0.8mm,yshift=-0.8mm]tbl.south east);

  \begin{scope}[
    shift={([xshift=0.098\textwidth,yshift=-0.15\textwidth]tbl.north west)},
    x=0.40\textwidth,
    y=0.40\textwidth,
    scale=0.075,
    line cap=round,
    line join=round,
    every node/.style={font=\tiny, text=figbrown!80!black}
  ]

    % Triangle vertices
    \coordinate (pA) at (0,0);
    \coordinate (pB) at (5.2,0);
    \coordinate (pC) at (1.15,3.25);

    % Incenter and points of tangency
    \coordinate (pO) at (1.727,1.221);
    \coordinate (pD) at (1.727,0);
    \coordinate (pE) at (2.492,2.173);
    \coordinate (pF) at (0.576,1.628);

% Shaded pair of right triangles, as in the proof-without-words figure.
% Drawn a little more strongly so the shading remains visible over the table.
\fill[figbrown!26!parchment, opacity=0.92]
  (pO)--(pD)--(pB)--cycle;

\fill[figbrown!26!parchment, opacity=0.92]
  (pO)--(pE)--(pB)--cycle;

  % Incircle, made more visible with a pale underlay and darker top stroke.
\draw[
  parchment!90!white,
  line width=0.34mm,
  opacity=0.80
]
  (pO) circle[radius=1.221];

\draw[
  figbrown!85!black,
  line width=0.29mm,
  opacity=0.95
]
  (pO) circle[radius=1.221];

    % Outer triangle
    \draw[figbrown!82!black, line width=0.20mm]
      (pA)--(pB)--(pC)--cycle;

    % Segments from incenter
    \draw[figbrown!82!black, line width=0.17mm] (pO)--(pA);
    \draw[figbrown!82!black, line width=0.17mm] (pO)--(pB);
    \draw[figbrown!82!black, line width=0.17mm] (pO)--(pC);

    % Inradii to the three sides
    \draw[figbrown!70, line width=0.10mm] (pO)--(pD);
    \draw[figbrown!70, line width=0.10mm] (pO)--(pE);
    \draw[figbrown!70, line width=0.10mm] (pO)--(pF);

    % Tangency chords/partition lines
    \draw[figbrown!70, line width=0.10mm] (pF)--(pE);
    \draw[figbrown!70, line width=0.10mm] (pF)--(pD);

    % Small right-angle mark at the bottom tangency point
    \draw[figbrown!70, line width=0.09mm]
      ($(pD)+(-0.16,0)$)--($(pD)+(-0.16,0.16)$)--($(pD)+(0,0.16)$);

    % Labels for tangent lengths
    \node[below] at ($(pA)!0.50!(pD)$) {$x$};
    \node[below] at ($(pD)!0.52!(pB)$) {$y$};

    \node[left]  at ($(pA)!0.52!(pF)$) {$x$};
    \node[left]  at ($(pF)!0.52!(pC)$) {$z$};

    \node[right] at ($(pB)!0.52!(pE)$) {$y$};
    \node[right] at ($(pE)!0.52!(pC)$) {$z$};

    % Inradius label
    \node[right] at ($(pO)!0.55!(pD)$) {$r$};

  \end{scope}
\end{scope}

\end{tikzpicture}

% Reset table rule color for later tables, just in case.
\arrayrulecolor{black}

\end{wraptable}

\section{Morning}

\subsection{Background: the peril of starting to teach at noon} 

In a Monday class, during the fall semester of 2025, the class and I used Heron's formula for the area of a triangle in terms of its side lengths to give a visual, calculus-free proof of the area-maximizing property of equilateral triangles among all triangles of perimeter $1$. We motivated and derived Heron's formula using our preferred presentation \cite{klain-2004}, which we recommend for inclusion in the next edition of \emph{Proofs From THE BOOK} \cite{proofs_from_z_book-2018}. I woke up on the following rainy Wednesday morning with a nagging thought: yeah, yeah, one can motivate and derive Heron's formula, but how would one actually \emph{discover} it? Heron's proof was quite involved\footnote{See the foreground of Table~\ref{tab:triangles} for a replica of a diagram replicated from a copy of Heron's 60 CE book \emph{Metrica}\cite{bruins-1964}, and an image from a much simpler \emph{proofs without words} approach\cite{nelsen-2001}. We'll let the reader guess which is which.}, not likely paralleling a path to discovery. Could we do a ``laboratory" experiment to discover the formula? Could I design one in time for today's noon class? This is one of the perils of starting to teach at noon. Same-day lesson plan alterations are out of the question for a 9am class, but with hours to spare, and AI at hand\footnote{Initial tables and charts were produced by \emph{Claude}, later checked by \emph{ChatGPT} and \emph{Gemini}, and then by a human. \label{fn:whichAI}}, hmmm, maybe something could be done...

My commute involves a fair bit of walking, combined with a tram and a subway ride. I set out with my mobile phone in one hand and rain gear in the other, stopping to converse with an AI in rain-shelters, on train platforms, and on rides. The interaction style was influenced by E.~Mollick's approach to collaboration \cite{mollick-2024}.

\subsection{The Gold Standard}

I started by asking the AI for a table of twenty arbitrary, quasi-random triangles, listed with their side lengths and their areas, with areas not computed via Heron's formula. To ground the experiment in a pre-computational spirit, I imagined monks carefully cutting triangles from parchment, then measuring base length and height with great care. This table would be our gold standard. The choice of twenty as the number of triangles to use came from intuition and turned out to be sufficient to do the job. The initial table included a few rows that did not satisfy the triangle inequality, as a check by another AI revealed. This was fixed.

\subsection{Fit for a Quadratic}

\begin{wrapfigure}[22]{r}{8.2cm} % r for right, width of the wrapped area
\centering
\vspace{-10pt} % Optional: adjust vertical alignment
\begin{tikzpicture} \label{quadratic_fit}
\begin{axis}[
    width=7.5cm, % Reduced width to fit the wrap
    height=5cm,
    xlabel={},
    ylabel={Area},
    title={True Area vs. Quadratic  Regression},
    xtick={1,5,10,15,20}, % Simplified ticks for space
    xmin=0, xmax=21,
    ymin=0, ymax=2.8,
    legend style={at={(0.03,0.97)}, anchor=north west, font=\tiny, fill=white, fill opacity=0.8},
    grid=major,
    grid style={gray!30},
]
\addplot[only marks, 
    mark=*, 
    mark size=2.5pt, 
    mark options={fill=gold, draw=black, line width=0.6pt}] coordinates {
    (1, 1.154186) (2, 0.963894) (3, 0.646708) (4, 0.482996) (5, 2.049419)
    (6, 1.138238) (7, 2.459089) (8, 0.546850) (9, 0.811159) (10, 0.587580)
    (11, 0.700776) (12, 1.586595) (13, 1.332507) (14, 1.066113) (15, 1.080001)
    (16, 1.166371) (17, 0.993562) (18, 1.684140) (19, 1.152307) (20, 0.507586)
};
\addlegendentry{Actual area}
\addplot[only marks, mark=x, mark size=2.5pt, color=blue, thick] coordinates {
    (1, 1.048977) (2, 0.902595) (3, 0.332716) (4, 0.585775) (5, 1.978823)
    (6, 1.341786) (7, 2.373556) (8, 0.451449) (9, 0.831833) (10, 0.814128)
    (11, 0.958010) (12, 1.534520) (13, 1.253958) (14, 0.899088) (15, 1.060838)
    (16, 1.204169) (17, 1.084992) (18, 1.580475) (19, 1.079833) (20, 0.874480)
};
\addlegendentry{Quad Approx}
\end{axis}
\end{tikzpicture}

\vspace{0.2cm}

\begin{tikzpicture}
\begin{axis}[
    width=7.5cm, % Reduced width
    height=4cm,
    xlabel={Triangle index},
    ylabel={{\scriptsize Error (Actual $-$ Predicted)}},
    xtick={1,5,10,15,20},
    xmin=0, xmax=21,
    ymin=-0.45, ymax=0.45,
    grid=major,
    grid style={gray!30},
    ybar,
    bar width=4pt,
]
\addplot[fill=purple!70, draw=purple!50!black] coordinates {
    (1, 0.105209)  (2, 0.061299)  (3, 0.313992)  (4, -0.102779) (5, 0.070596)
    (6, -0.203548) (7, 0.085533)  (8, 0.095401)  (9, -0.020674) (10, -0.226548)
    (11, -0.257234)(12, 0.052075) (13, 0.078549)  (14, 0.167025) (15, 0.019163)
    (16, -0.037798)(17, -0.091430)(18, 0.103665)  (19, 0.072474) (20, -0.366894)
};
\end{axis}
\end{tikzpicture}
\caption{Quadratic fit results.}
\label{fig:quadratic_wrap}
\vspace{-10pt}
\end{wrapfigure}
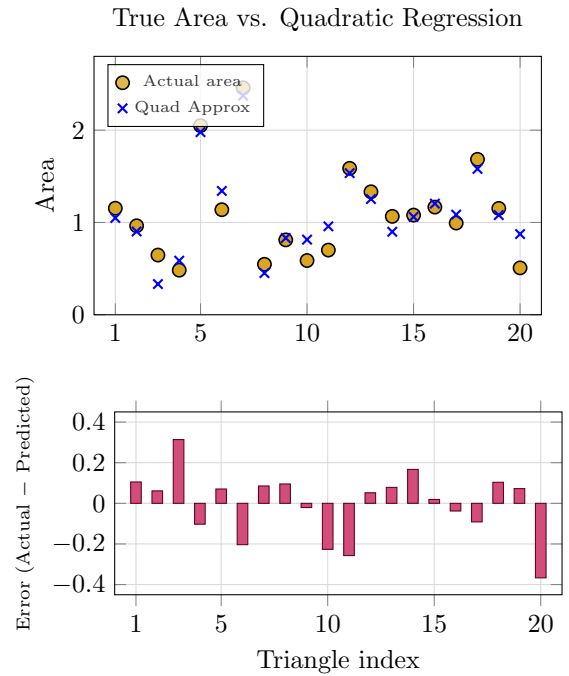

Next, I asked the AI to do a \emph{quadratic} least squares fit of the triangle area function. My rationale was that the perimeter of a triangle with side lengths $\{a, b, c\}$ is linear in $a, b, c$, so I expected that in the lab I would surmise that the area should be quadratic. Of course, we aimed for a homogeneous polynomial, symmetric in $a,b,c$.

\begin{quoting}[leftmargin=0.5em, rightmargin=0.5em]
{\footnotesize I made a mental note to pause during class, ask students to plot a few random points on the parabola $y=x^2$, then make a linear regression fit and notice that a line does not do a good job approximating a parabola. (This would also provide a good opportunity to review the concept of \emph{least squares fit}.) I decided not to include this step in the AI conversation, but to save it for ``winging-it" in class. \par}
\end{quoting}

The results came back: the fit was fail-to-middling for some triangles, good for others. The failure is not merely numerical. It reflects a structural limitation.

\begin{proposition}
There is no quadratic polynomial in side lengths $a,b,c$ that expresses the area of $\triangle abc$.
\end{proposition}
\begin{proof}
If a quadratic gave the exact area for all triangles, the least-squares fit would recover it exactly.
\end{proof}

\begin{wrapfigure}{r}{
%0.48\textwidth
9cm}
\vspace{-1.2\baselineskip}
\centering

\begin{tikzpicture}
\begin{axis}[
    width=9.5cm,
    height=7cm,
    xlabel={},
    ylabel={Area},
    title={{ Actual Area vs.\  Quartic} Approx (fitting Area$^2$)},
    xtick={1,2,3,4,5,6,7,8,9,10,11,12,13,14,15,16,17,18,19,20},
    xticklabels={},
    xmin=0, xmax=21,
    ymin=0, ymax=2.9,
    legend style={at={(0.03,0.97)}, anchor=north west, fill=white, fill opacity=0.8},
    grid=major,
    grid style={gray!30},
]
\addplot[only marks, mark=*, mark size=3.5pt, color=gold] coordinates {
    (1, 1.154186) (2, 0.963894) (3, 0.646708) (4, 0.482996) (5, 2.049419)
    (6, 1.138238) (7, 2.459089) (8, 0.546850) (9, 0.811159) (10, 0.587580)
    (11, 0.700776) (12, 1.586595) (13, 1.332507) (14, 1.066113) (15, 1.080001)
    (16, 1.166371) (17, 0.993562) (18, 1.684140) (19, 1.152307) (20, 0.507586)
};
\addlegendentry{Actual area}
\addplot[only marks, mark=x, mark size=4pt, color=blue, thick] coordinates {
    (1, 1.154186)  (2, 0.963894)  (3, 0.646708)  (4, 0.482996)  (5, 2.049419)
    (6, 1.138238)  (7, 2.459089)  (8, 0.546850)  (9, 0.811159)  (10, 0.587581)
    (11, 0.700776) (12, 1.586595) (13, 1.332507) (14, 1.066113) (15, 1.080001)
    (16, 1.166371) (17, 0.993562) (18, 1.684140) (19, 1.152307) (20, 0.507587)
};
\addlegendentry{Quartic Poly Approx}
\end{axis}
\end{tikzpicture}

\vspace{-0.3cm}

\begin{tikzpicture}
\begin{axis}[
    width=9cm,
    height=5cm,
    xlabel={Triangle index},
    ylabel={\scriptsize Error (Actual $-$ Predicted)},
    xtick={1,2,3,4,5,6,7,8,9,10,11,12,13,14,15,16,17,18,19,20}, 
    xmin=0, xmax=21,
    ymin=-0.1, ymax=0.1,
    % --- Fixed y-axis formatting ---
    ytick={-0.1, -0.05, 0, 0.05, 0.1},
    scaled y ticks=false,
    yticklabel style={
        /pgf/number format/fixed,
        /pgf/number format/precision=2
    },
    % -------------------------------
    grid=major,
    grid style={gray!30},
    ybar,
    bar width=8pt,
]
\addplot[fill=purple!70, draw=purple!50!black] coordinates {
    (1, 0.000000)  (2, -0.000000) (3, 0.000000)  (4, 0.000000)  (5, -0.000000)
    (6, 0.000000)  (7, 0.000000)  (8, -0.000000) (9, -0.000000) (10, -0.000001)
    (11, -0.000000)(12, -0.000000)(13, 0.000000)  (14, -0.000000)(15, 0.000000)
    (16, 0.000000) (17, -0.000000)(18, -0.000000) (19, 0.000000) (20, -0.000001)
};
\end{axis}
\end{tikzpicture}
\vspace{0.2cm}
%\noindent
{ \scriptsize 
$-0.0625(a^4 + b^4 + c^4) + 0.1250\,(a^2b^2 + b^2c^2 + c^2a^2)$, (exact fit)}
\caption{Quartic fit---perfect.}
\label{fig:quartic}
\vspace{-0.8\baselineskip}
\end{wrapfigure}

\subsection{Quartic Quest}
If a quadratic doesn't work, we can look for a \emph{quartic}, a fourth-order polynomial. Area scales quadratically under dilation, so quartics would model the \emph{square} of the area. I started to ask the AI for a quartic fit, slightly insecure in the feeling that twenty triangles might not suffice.

\begin{quoting}[leftmargin=0.5em, rightmargin=0.5em]
{\footnotesize I made a mental note about \emph{why quartic}? Surely, someone in class will ask why we skipped cubics. The rationale is that area transforms quadratically under dilations. If we go to quartics, we'd expect an expression for the square of an area---transforming quartically under dilations. The Pythagorean theorem relates squares of side lengths, so why not a relation among squares of areas? On the other hand, a degree $3$ polynomial transforms cubically under dilations, requiring radicals to relate to the quadratic-transforming area; that's reason enough for omission of cubics.
\par }
\end{quoting}

Asking the AI for a quartic fit went fine---twenty triangles sufficed. The results were decisive: a perfect fit. Of course, this might only work for our twenty triangles, and not for others. (In fact, the resulting quartic coincides with the classical expression equivalent to Heron’s formula.) But no matter, we have made a discovery. We don't need to retrieve the actual polynomial. We can now make a conjecture, deliberately weaker than Heron's formula itself, and intuitive derivations \cite{klain-2004},\cite{nelsen-2001} can supply the rest.

\begin{conjecture}
The squared area of a triangle $\triangle abc$ is a homogeneous quartic polynomial, symmetric in $a,b,c$. \end{conjecture}

\subsection{Historical Context}
Figure~\ref{figH}, probably inspired by a dream the night before and also by the story \cite{clarke-1967}, is not historical. Least Squares Fit didn't come for millennia after Heron. But it could fit in the early 19th century. By then, non-Euclidean geometries had emerged. People may have wondered if analogs of Heron's formula exist in these geometries. To discover them, they might have first engineered a way to discover the Heron formula of Euclidean geometry.  

Also by that time, Lagrange had connected determinants with volumes (1773), and Cauchy and Binet established the determinant product law (1812). Guyton de Morveau introduced the graduated glass cylinder (1784), and Legendre (1805) and Gauss (1795--1810) introduced least squares fit. Thus monks were in a position to discover the Heron formula in the plane and its generalization to tetrahedra in space, and the ingredients for intuitive derivations \cite{klain-2004} were largely available. So the monks could discover that a sextic polynomial in edge lengths expresses the squared volume  of a tetrahedron, setting the stage for the Cayley (1841) and Menger (1928) determinant generalizations.

\section{Afternoon and Aftermath}

\subsection{Back to the Future}
Least Squares is not the only means for going back in time, using modern tools in ancient venues. M.~Buysse \cite{buysse-2023} takes Isaac Newton on a tour of ancient geometric contexts. For Heron's formula, Isaac assumes that $\operatorname{Area}\triangle(x,y,z)$ is a smooth function of the side lengths $x,y,z$. He considers a variation $\delta x$ and develops a partial differential equation for the area function, which involves side lengths. Integrating the equation and exploiting symmetry, he obtains Heron's formula. This Calculus of Variations approach is broadly effective, as \cite{buysse-2023} shows.

In another variation, \cite{pythagoras_via_cavalieri-2015} uses the celebrated Cavalieri Principle (mid 17th century, possibly known to Archimedes) to prove the Pythagorean Theorem. The historic reversal is justified by citing \cite{pythagoras_by_heron-2010}, which proves Pythogoras by Heron. We recommend asking students to work out the details of the argument in \cite{pythagoras_via_cavalieri-2015}, for an effective hands-on experience of the magic of the Cavalieri principle, one that may leave a more lasting impression that the typical textbook treatment.

\subsection{Generalizations} 
Does the Least Squares strategy for triangle areas extend to higher dimensions? A single subway ride was sufficient for choosing 60 generic tetrahedra, computing their volumes by a geometric vector approach, modeling  squared volumes by cubics in squares of the six edge lengths, and finding the best cubic using Least Squares fit. The result was the Heron-Tartaglia formula \cite{sabitov-1998}, displayed here compactly with scalable, easily magnified font:\\ \\
{\tiny
\(
V^2 = \frac{1}{144} \Big( \ell_1^2 \ell_5^2(\ell_2^2 + \ell_3^2 + \ell_4^2 + \ell_6^2 - \ell_1^2 - \ell_5^2) + \ell_2^2 \ell_6^2(\ell_1^2 + \ell_3^2 + \ell_4^2 + \ell_5^2 - \ell_2^2 - \ell_6^2) + \ell_3^2 \ell_4^2(\ell_1^2 + \ell_2^2 + \ell_5^2 + \ell_6^2 - \ell_3^2 - \ell_4^2) - \ell_1^2 \ell_2^2 \ell_3^2 - \ell_1^2 \ell_4^2 \ell_6^2 - \ell_2^2 \ell_4^2 \ell_5^2 - \ell_3^2 \ell_5^2 \ell_6^2 \Big)
\).
} \\

\noindent
(The enumeration of tetrahedral edges $\ell_i$ requires extra care \cite{sabitov-1998}.) \\

How about alternative geometries? Replacing the Euclidean norm by the Taxicab $L^1$ geometry, is there an analog of Heron-Tartaglia? Not if only edge lengths are used, but yes if some additional geometric invariants are introduced \cite{colakoglu-2009}. Can this be discovered via Least Squares? Alas, the subway ride was too short to handle the Taxicab case. \\

In addition to varying the ambient dimension one can also vary subspace dimensions. T.F.~Havel \cite{Havel-2023} gives a formula for the fourth power of the volume of a tetrahedron using a polynomial in areas of facets and of ``medial sections". Discovering that facet areas alone do not suffice is probably straightforward, but discovering Havel's formula would be considerably more challenging: one would first have to ``discover" appropriate supplementary measurements, such as medial section areas.

\subsection{Symmetry}
Together with varying dimensions, one can also restrict the context by imposing conditions of symmetry. Klain gives intuitive derivations of simpler formulas for \emph{isosceles} tetrahedra \cite{klain-2004}, and for \emph{reversible} tetrahedra \cite{klain-2023}; the assumed symmetries are congruence of all facets, and congruence of pairs of facets, respectively. These formulas factor, and the factors have geometric interpretations.  Once an aforementioned symmetry is imposed, discovering a volume formula by Least Squares seems attainable. One has the added challenge of generating arbitrary tetrahedra subject to a symmetry restriction, but in these contexts this seems feasible.

\subsection{Going Beyond Length} The volume and area  formulas that we (re)discovered so far are all length-based. But we can apply Least Squares to discover formulas of other types as well. Using John H. Conway's triangle notation, we write $S_A$ for the inner product of the vectors emanating from the vertex $A$ of $\triangle ABC$, and similarly for $S_B,S_C$. Then we have the formula
\[
\left(\operatorname{Area} \triangle ABC \right)^2
= \frac{1}{4} \left( S_A S_B + S_B S_C + S_C S_A \right).
\]

\begin{wrapfigure}[8]{r}{0.1\textwidth}
\includegraphics[width=0.1\textwidth]{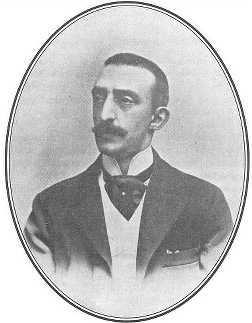}
\caption*{\tiny J.~Dur\'an Loriga}
\end{wrapfigure}

\, \\
This formula, and the triangle notation popularized by Conway, can be found in J.~Dur\'an Loriga's 1894 paper \cite{duran-loriga-1894}. We were pleasantly surprised to see that a scan of the original paper is available from the National Library of Spain. Could we obtain the inner products area formula by Least Squares? Probably, but let's make a spatial analog of \emph{Back To The Future}: work out a 3-dimensional case, then look back to the plane.
According to M.~Hajja and H.~Martini \cite{hajja-martini-2013}, the true generalization of triangles
to higher dimensions is not general simplices but rather \emph{orthocentric simplices}. These are tetrahedra whose four altitudes meet at a point, a property not satisfied by generic tetrahedra. Equivalently, an orthocentric simplex is a tetrahedron for which the inner product of any two vectors emanating from a given vertex $A$ is the same for all pairs of vectors so emanating, and that's the property we'll use. \\

\vspace{-5pt}
\begin{wrapfigure}[7]{l}{0.30\textwidth}
  \centering
  \begin{tikzpicture}[scale=0.30, line join=round, line cap=round]
    
    % Coordinates
    \coordinate (A) at (0, 3.8);
    \coordinate (B) at (-3.5, -1.8);
    \coordinate (C) at (1, -3.2);
    \coordinate (D) at (3.8, -0.8);
    
    % Point H on CD for the altitudes
    \coordinate (H) at (2.96, -1.52);

    % Background Fill
    \fill[gray!8] (A) -- (B) -- (C) -- (D) -- cycle;

    % Dashed Lines
    \draw[dashed, semithick] (B) -- (D);
    \draw[dashed, semithick] (A) -- (H);
 %  \draw[dashed, semithick] (B) -- (H);

    % Solid Base Lines
    \draw[semithick] (B) -- (C) -- (D);

    % Vectors
    \draw[->, ultra thick, >=Latex] (A) -- (B) node[midway, above left, xshift=0.2cm , yshift=0.4cm] 
    {\footnotesize $\vec{v}_1 = \vec{AB}$};
    \draw[->, ultra thick, >=Latex] (A) -- (C) node[midway, left=-0.1cm, yshift=-0.3cm] 
    {\footnotesize $\vec{v}_3 = \vec{AC}$};
    \draw[->, ultra thick, >=Latex] (A) -- (D) node[midway, above right, xshift=-0.1cm] 
    {\footnotesize$\vec{v}_2 = \vec{AD}$};

    % Vertex Labels
    \node[above]       at (A) {\large $A$};
    \node[left]        at (B) {\large $B$};
    \node[below right] at (C) {\large $C$};
    \node[right]       at (D) {\large $D$};

  \end{tikzpicture}
\end{wrapfigure}

We want to generate 60 arbitrary orthocentric tetrahedra. We can do this by choosing two vectors, $\vec v_1 , \vec v_2$,  emanating from $A$ at random (with integer lengths for computational ease), and then we need to choose $\vec v_3$ so that the inner products $< \vec v_3, \vec v_1>$ and $< \vec v_3, v_2>$ are equal to $< \vec v_1, v_2>$. This turns out to ensure that the resulting simplex will close up and be orthocentric. Such simplices have the property that any pair of edge vectors emanating from the same vertex has the same inner product as any other pair. So we take the inner products $S_A,S_B,S_C,S_D$, using Conway-Dur\'an Loriga notation, and consider cubic polynomials in them as candidates for the square of the volume. Least squares fit gives the following, perhaps surprisingly simple expression:
\[  
%\boxed{\,
V^2=\tfrac1{36}\bigl(S_AS_BS_C+S_AS_BS_D+S_AS_CS_D+S_BS_CS_D\bigr)
%\,}
.
\]

We have not found this particular formula for the squared volume of an orthocentric tetrahedron in the literature, but it is routinely derivable: one starts from the Gram-determinant expression for the volume of a general tetrahedron (\cite{berger-1987} \S~9.7) and applies the simplifications afforded by orthocentricity (\cite{edmonds-hajja-martini-2005}, \S~3.1).
It is evident that its Dur\'an Loriga 2-dimensional analog is also discoverable by Least Squares.

\subsection{Looking Back}
Reflecting on the experience, the class presentation probably would not have taken place if not for easy AI access. I might have outlined the idea at the board, but I would have had neither the areas table \ref{tab:triangles}, nor the regression plots \ref{fig:quadratic_wrap}, \ref{fig:quartic}. Producing these sans AI would have taken much longer, and might have been foregone altogether. The availability of AI has profoundly shifted the turnaround time from vague idea to classroom-ready presentation. Next fall my class will begin at 11:00 instead of noon. Will AI improvements compensate?
\footnote{The title of this article is inspired by the book \emph{A Swim in a Pond in the Rain}, by George Saunders. The image was drawn in collaboration with Gemini AI, possibly influenced by a dream.}

\begin{figure}[H]
\hypertarget{fig:cartoon}{} % Create a target for the link
\caption{An Ahistorical Geometric Formula Discovery Laboratory \\ 
{\footnotesize (Image drawn in collaboration with Gemini AI, possibly influenced by a dream and a story \cite{clarke-1967})}}
\centering
\includegraphics[width=0.63 \textwidth]{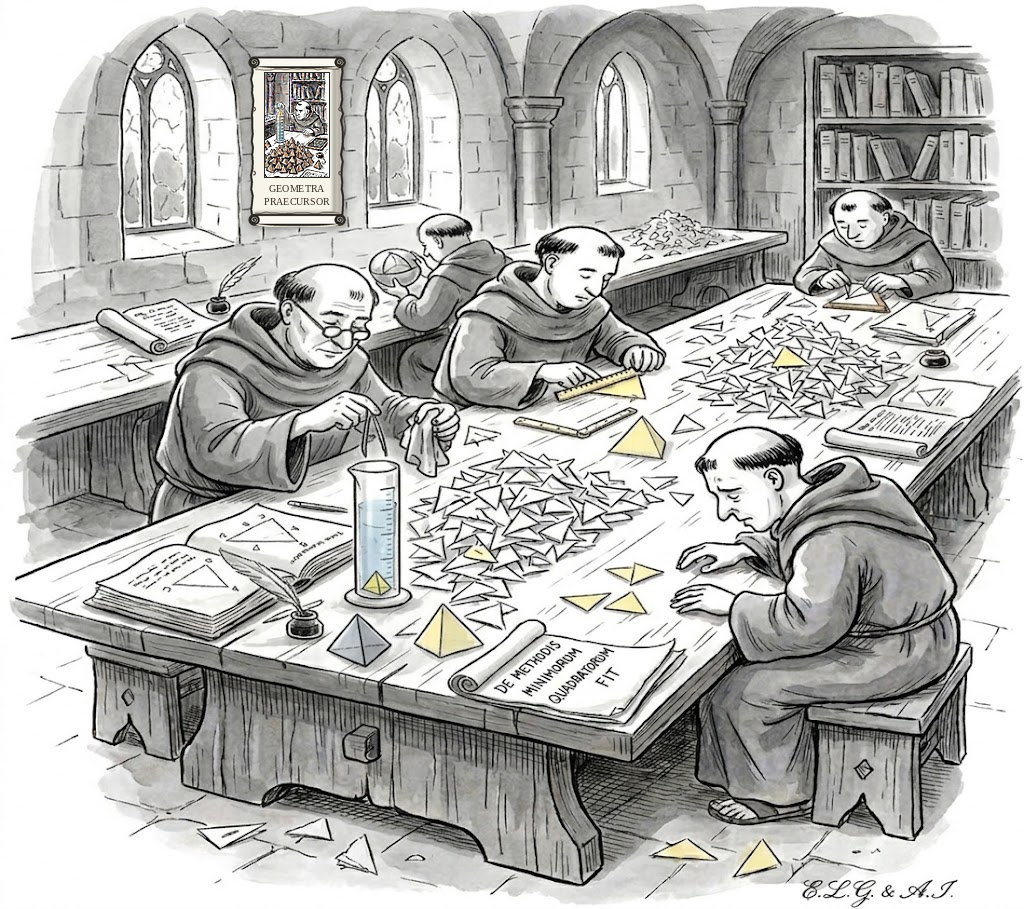}
\label{figH}
\end{figure}

\bigskip\noindent
Department of Mathematics, University of Massachusetts Boston,   USA \\
\qquad \href{mailto:eric.grinberg@umb.edu}{eric.grinberg@umb.edu} 

\bigskip

%\thanks{The title of this writing is in homage to George Saunders' book \emph{A Swim in a Pond in the Rain}.}

\end{document}